\documentclass[11pt]{article}
\RequirePackage{fix-cm}
\usepackage[utf8]{inputenc}

\usepackage{comment, amsmath, amscd, amsxtra, amssymb, amsthm, changepage, pifont, graphicx, xcolor, tikz, soul, longtable, listings, algorithm, enumitem, appendix, graphics, epsfig, epic, latexsym, booktabs, mathtools, colortbl, afterpage, stmaryrd, nomencl}
\usepackage[color]{changebar}
\usetikzlibrary{cd}

\usepackage[a4paper, total={6in, 8in}]{geometry}

\usetikzlibrary{calc, math}

\usepackage[disable]{todonotes}

\allowdisplaybreaks

\numberwithin{figure}{section}

\newtheorem{theorem}{Theorem}[section]
\newtheorem{lemma}[theorem]{Lemma}
\newtheorem{proposition}[theorem]{Proposition}

\theoremstyle{definition}
\newtheorem{definition}[theorem]{Definition}

\theoremstyle{remark}
\newtheorem{remark}[theorem]{Remark}

\DeclareMathOperator{\charac}{char}

\DeclareMathOperator{\im}{im}

\newcommand{\ZZ}{\mathbb{Z}}
\newcommand{\QQ}{\mathbb{Q}}
\newcommand{\RR}{\mathbb{R}}

\newcommand{\FF}{\mathbb{F}}

\renewcommand{\AA}{\mathbb{A}}
\renewcommand{\epsilon}{\varepsilon}
\newcommand{\PP}{\mathbb{P}}

\title{\texorpdfstring{Points of bounded height on curves and the dimension growth conjecture over $\FF_q[t]$}{Points of bounded height on curves and the dimension growth conjecture over Fq[t]}}
\author{Floris Vermeulen}
\date{24/03/2020}

\begin{document}

\maketitle

\begin{abstract}
In this article we prove several new uniform upper bounds on the number of points of bounded height on varieties over $\FF_q[t]$. For projective curves, we prove the analogue of Walsh' result with polynomial dependence on $q$ and the degree $d$ of the curve. For affine curves, this yields an improvement to bounds by Sedunova, and Cluckers, Forey and Loeser. In higher dimensions, we prove a version of dimension growth for hypersurfaces of degree $d\geq 64$, building on work by Castryck, Cluckers, Dittmann and Nguyen in characteristic zero. These bounds depend polynomially on $q$ and $d$, and it is this dependence which simplifies the treatment of the dimension growth conjecture. 
\end{abstract}

\section{Introduction}

For a tuple of integers $a=(a_1, ..., a_n)\in \ZZ^n$ recall that the height is defined to be $H(a) = \max_i |a_i|$. If $X$ is some subset of $\AA^n_\QQ$ we are interested in the counting function
\[
N_\mathrm{aff}(X; B) = \#\{a \in \ZZ^n\mid a\in X, H(a)\leq B\}.
\]
In particular, we want to find upper bounds on this quantity which are uniform with respect to the particular set $X$. Around 30 years ago, Bombieri and Pila \cite{Bombieri-Pila} pioneered the fruitful determinant method to study this counting function for various $X\subseteq \RR^2$. Their method is based on Taylor approximation. For an integral algebraic curve $C$ of degree $d$ this yields for any $\epsilon>0$ the bound
\[
N_\mathrm{aff}(C; B)\leq_{\epsilon, d}B^{\frac{1}{d}+\epsilon}.
\]
It should be noted that the implicit constant depends only on the degree of $C$, and not on $C$ itself. 

The determinant method was subsequently improved by Heath-Brown \cite{Heath-Brown-Ann}, who replaced the Taylor approximation from Bombieri and Pila by $p$-adic approximation, and later by Salberger \cite{Salberger-dgc}, who developed a global version of this $p$-adic determinant method. Heath-Brown and Salberger were able to prove results for higher dimensional varieties in this way. 

In the projective setting, if $a=(a_0: ...: a_n)$ is a point in $\PP^n_\QQ$ we may assume that the $a_i$ are coprime integers and then we define the height to be $H(a) = \max_i |a_i|$. If $X$ is a subset of $\PP^n$ then we define similarly the counting function
\[
N(X; B) = \#\{a\in \PP^n(\QQ) \mid a\in X, H(a)\leq B\}.
\]
For projective curves, Walsh \cite{Walsh} has proven that an integral projective curve $C$ in $\PP^n$ of degree $d$ satisfies
\[
N(C; B)\leq_{d,n}B^{\frac{2}{d}}.
\]
Again, the constant here only depends on the degree of $C$ and not on $C$ itself. Recently, Castryck, Cluckers, Dittmann and Nguyen \cite{CCDN} have been able to make the dependence on $d$ explicit, proving that for an irreducible projective curve $C\subseteq \PP^n$ of degree $d$ we have
\[
N(C; B)\leq_n d^4B^{\frac{2}{d}}.
\]

For a higher dimensional integral projective variety $X$, one cannot possibly hope for a bound of the form $B^{2/d}$ (or $q^{2\ell/d}$ in positive characteristic), because of the existence of linear spaces contained in $X$. Instead, Serre \cite{Serre-Mordell} and Heath-Brown \cite{Heath-Brown-cubic} formulate the dimension growth conjecture, which asserts that if $d\geq 2$ then
\[
N(X; B)\leq_{d, n, \epsilon} B^{\dim X+\epsilon}.
\]
Dimension growth is now a theorem by work of Browning, Heath-Brown and Salberger for $d\geq 6$ \cite{Brow-Heath-Salb} and later for $d\geq 4$ by work of Salberger \cite{Salberger-dgc}. For $d\geq 5$, this was recently improved to \cite{CCDN}
\[
N(X; B)\leq_{n} d^{O_n(1)}B^{\dim X}.
\]
One of the central ideas in that work is that keeping track of the polynomial dependence on $d$ heavily simplifies the dimension growth conjecture for $d\geq 16$. 

This article is concerned with the $\FF_q[t]$-analogue of these questions. Let us first set some notation. For $\alpha\in \FF_q[t]$ we define the norm of $\alpha$ to be $|\alpha| = q^{\deg\alpha}$. For a point $\alpha = (\alpha_0: ...: \alpha_n)\in \PP^{n}(\FF_q(t))$ represented by a tuple of coprime polynomials in $\FF_q[t]$ we define the height in the same fashion to be $H(\alpha) = \max_i |\alpha_i|$. Similarly, we define the logarithmic height by $h(\alpha) = \max_i \log_q|\alpha_i| = \max_i \deg \alpha_i$. Typically, it is more natural to work with this logarithmic height function. For $X$ a subset of $\PP^n_{\FF_q(t)}$ we define the counting function as
\[
N(X;\ell) = \#\{x\in \PP^{n}(\FF_q(t))\mid x\in X, h(x)<\ell \}.
\]
If $f\in \FF_q[t][x_0, ..., x_n]$ is some homogeneous polynomial, we define $N(f; \ell)=N(V(f); \ell)$.

In the affine setting, for a tuple $\alpha=(\alpha_1, ..., \alpha_n)\in \FF_q[t]^n$ we define the height as $H(\alpha) = \max_i|\alpha_i|$ and the logarithmic height as $h(\alpha) = \log_qH(\alpha)$. For $X\subseteq\AA^n_{\FF_q(t)}$ we define the affine counting function
\[
N_\mathrm{aff}(X; \ell) = \#\{x\in \FF_q[t]^n \mid x\in X, h(x)<\ell\}.
\]
Finally, if $f\in \FF_q[t][x_1, ..., x_n]$ is a polynomial then we define $N_\mathrm{aff}(f; \ell) = N_\mathrm{aff}(V(f); \ell)$.

The available techniques and hence the final results differ in small and large characteristic. Large characteristic fields behave much more like characteristic zero, and so we are able to obtain better results in this way. Let $c=\charac \FF_q$ and fix a variety $X\subseteq \PP^n$ of degree $d$, we will say that we are in \emph{small characteristic} if $c\leq d(d-1)$, in \emph{large characteristic} if $c> d(d-1)$ and in \emph{very large characteristic} if $c^{1-\epsilon}> 27d^4$ (for some fixed $\epsilon>0$). For real numbers $a,b,c$ we define the quantity $[a,b,c]$ to be $a$ in small characteristic, $b$ in large characteristic and $c$ in very large characteristic. 

Our first main result is the following upper bound on $N(f; \ell)$ for projective plane curves.

\begin{theorem}\label{thm: projective counting}
Let $f\in \FF_q[t][x_0, x_1, x_2]$ be an irreducible primitive homogeneous polynomial of degree $d$. Then
\[
N(f; \ell) \leq_\epsilon q^{1+\epsilon}d^{[8, \frac{14}{3}, 4]} q^{\frac{2(\ell-1)}{d}}.
\]
\end{theorem}

A more precise statement depending on $f$ can be deduced from Theorem \ref{thm: main theorem}. For affine curves this implies the following result, by using a technique from \cite{Ellenb-Venkatesh}.

\begin{theorem}\label{thm: affine counting}
Let $f\in \FF_q[t][x_1, x_2]$ be a primitive irreducible polynomial of degree $d$. Then 
\[
N_\mathrm{aff}(f; \ell) \leq_\epsilon q^{1+\epsilon}d^{[7,3,3]}q^{\frac{(\ell-1)}{d}} (\ell+d^{[1, \frac{5}{3}, 1]}).
\]
\end{theorem}

This improves upon a result by Sedunova \cite{sedunova} which states that 
\[
N_\mathrm{aff}(f; \ell)\leq_{\epsilon, q, d}q^{\ell(\frac{1}{d}+\epsilon)},
\]
by making the constant explicit in terms of $q$ and $d$ and replacing the $\epsilon$ by a logarithmic term. Recently, Cluckers, Forey and Loeser \cite{cluckers2019uniform} have proven a similar result stating that for any $d$, there is a constant $c_d$ such that in characteristic $>c_d$ we have
\[
N_\mathrm{aff}(f; \ell)\leq_d \ell^2 q^{\left\lceil\frac{\ell}{d} \right\rceil}.
\]
The restriction on characteristic comes from the fact that their proof is based on model theoretic methods. Note that the dependence on $q$ here is at most $q^{1-1/d}$. Thus our dependence on $q$ is slightly worse, but we remove one of the logarithmic factors. 

In higher dimensions, we prove the first dimension growth result in positive characteristic. This gives an upper bound on $N(X;\ell)$ for projective hypersurfaces with polynomial dependence on $d$ and $q$. Using projection arguments similar to \cite[Sec.\,5]{CCDN} it should be possible to extend this result to all integral projective varieties and not just hypersurfaces.

\begin{theorem}\label{thm: projective dimension growth}
Let $X\subseteq \PP^n$ be an integral projective hypersurface of degree $d\geq 64$ defined over $\FF_q[t]$. Then for any $\ell\geq 1$
\[
N(X; \ell)\leq_{\epsilon, n} q^{9+\epsilon}d^{O_n(1)}q^{\ell(n-1)}.
\]
If moreover, $\ell\geq 10$ and $\frac{\ell-1}{\ell-9}\leq \frac{\sqrt{d}}{8}$ then we have the stronger bound
\[
N(X; \ell)\leq_{\epsilon, n} q^{1+\epsilon}d^{O_n(1)}q^{\ell(n-1)}.
\]
\end{theorem}

Our proof is based on a similar result for affine hypersurfaces. At a crucial point, we will be applying Theorem \ref{thm: affine counting} to certain auxiliary curves lying on the hypersurface. The factor $q^9$ is then naturally explained by the corresponding $d^8$ in the small characteristic case of Theorem \ref{thm: affine counting}. Unfortunately, the degrees of these curves can become very large with respect to the characteristic, and so we are unable to improve this result in large characteristic.

Our method relies on a reworked version of the determinant method in positive characteristic and is heavily based upon the techniques in \cite{CCDN}. One of the key points is that keeping track of polynomial dependence on $d$ in all bounds developed allows us to give a simplified treatment of the dimension growth conjecture for $d\geq 64$. We have decided to include full proofs of most results to keep track of the dependence on $q$ and to stress the differences.

\paragraph{Notation.} For $f\in \FF_q[t][x_0, ..., x_n]$ we define $||f||$ to be the maximum norm of the coefficients of $f$. For a positive integer $k$ we define
\[
\FF_q[t]_{<k} = \{\alpha\in \FF_q[t]\mid \deg \alpha < k\}
\]
and similarly for $\FF_q[t]_{\leq k}$. If $X$ is a variety over $\FF_q[t]$ and $p$ is a prime in $\FF_q[t]$ then we denote by $X_p$ the reduction of $X$ modulo $p$. We use the standard asymptotic notation $f=O(g)$ to mean that $|f|\leq C|g|$ for some constant $C$. If $C$ depends on certain parameters then this will be indicated by a subscript. By a prime we will always mean a monic irreducible polynomial in $\FF_q[t]$.

We will often make implicit use of the prime number theorem over $\FF_q[t]$, which states that if $N_n$ denotes the number of primes of degree $n$ in $\FF_q[t]$ then
\[
N_n = \frac{q^n}{n} + O\left(\frac{q^{n/2}}{n}\right),
\]
where the constant is absolute. The error term here is the analogue of the Riemann-hypothesis.

\paragraph{Acknowledgements.} The author would like to thank Wouter Castryck, Raf Cluckers, Arthur Forey and Kien Huu Nguyen for helpful discussions.

\section{Preliminary estimates}

\paragraph{A determinant estimate.} Let $X$ be a hypersurface in $\PP^{n+1}$ defined by a primitive absolutely irreducible homogeneous polynomial $f\in \FF_q[t][x_0, ..., x_{n+1}]$ of degree $d>1$. We need the following result, analogous to an estimate due to Salberger \cite{Salberger-dgc} which was further refined in \cite[Sec.\,2]{CCDN}. 

\begin{lemma}\label{thm: estimate one prime}
Let $p$ be a prime for which $X_p$ is absolutely irreducible and for which either $|p|\geq d^{14/3}$, or we are in very large characteristic. Let $\xi_1, ..., \xi_s$ be $\FF_q[t]$-points on $X$ and let $F_1, ..., F_s$ be homogeneous polynomials in $\FF_q[t][x_0, ..., x_{n+1}]$. Then the determinant of the matrix $(F_i(\xi_j))_{ij}$ is divisible by $p^e$ where 
\[
e\geq (n!)^{1/n}\frac{n}{n+1}\frac{s^{1+1/n}}{|p| + O_n(d^2|p|^{1/2})} - O_n(s).
\]
\end{lemma}
\begin{proof}
This is the main result of Section 2 of \cite{CCDN}. All of the results up to Proposition 2.6 are still valid in our context with near identical proofs. This yields the estimate
\[
e\geq (n!)^{1/n}\frac{n}{n+1} \frac{s^{1+1/n}}{n_p^{1/n}} - O_n(s),
\]
where $n_p$ is the number of points on $X_p$ counted with multiplicity. By \cite[Thm.\,5.2]{CafureMatera}, we obtain that the number of points on $X_p$ without multiplicity is at most
\[
\frac{1}{|p|-1}\left(|p|^{n+1}+(d-1)(d-2)|p|^{n+1/2}+5d^{13/3}|p|^{n}-1\right).
\]
By our assumption that $|p|\geq d^{14/3}$ this quantity is bounded by $|p|^n+O_n(d^2|p|^{n-1/2})$. In very large characteristic, we may instead apply \cite[Cor.\,5.6]{CafureMatera} to obtain the same estimate. Including the multiplicity in the count can be achieved in the same way as in \cite[Lem\,2.7]{CCDN}, thus obtaining
\[
n_p^{1/n}\leq |p|+O_n(d^2|p|^{1/2}).\qedhere
\]
\end{proof}

We will apply this result to a number of primes simultaneously. 

\begin{definition}
For $f$ as above, we define
\[
b(f) = \prod_p q^{\frac{\deg p}{|p|}},
\]
where the product is over those primes $p$ for which $\deg p > \lfloor \frac{14}{3}\log_q d \rfloor$ and $f\bmod p$ is not absolutely irreducible. We will put $\beta = \lfloor \frac{14}{3}\log_q d \rfloor \in \ZZ$. In very large characteristic, we instead take the product over all primes $p$ and put $\beta=0$.
\end{definition}

\begin{lemma}\label{thm: bound b(f)}
For any $\epsilon>0$, we have 
\begin{align*}
b(f) &\leq_\epsilon q^{\epsilon} \max\{ q^{-\beta} d^{[6,2,2]} \log_q||f||, 1\}.
\end{align*}
\end{lemma}
\begin{proof}
We use effective results on Noether polynomials, which differ for small and large characteristic. In small characteristic, we apply \cite[Thm.\,7]{KALTOFEN1995}, stating that there is a finite set of polynomials $(\Phi_i)_i$ defined over $\ZZ$ of degree at most $12d^6$ such that for a homogeneous polynomial $F$ over any field $K$ we have that $F$ is not absolutely irreducible if and only if all of the $\Phi_i$ vanish when applied to the coefficients of $F$. In large characteristic, we may instead use \cite[Satz.\,4]{RuppertCrelle} to obtain such polynomials of degree $d^2-1$, see also \cite{GAOpdes}.

Denote by $\mathcal{P}$ the set of primes $p$ with $\deg p>\beta$ for which $f\bmod p$ is not absolutely irreducible. Since $f$ is absolutely irreducible there is a $\Phi_i$ for which $\Phi_i(f)\neq 0$, but for which $\Phi_i(f\bmod p)=0$ for all $p\in \mathcal{P}$. In particular
\[
\prod_{p\in \mathcal{P}}|p|\leq ||f||^{\deg \Phi_i} =:c.
\]
Take $\delta\in [\epsilon, \epsilon +1)$ such that $\log_q\log_qc + \delta$ is an integer. Then
\begin{align*}
\log_q b(f) &= \sum_{p\in \mathcal{P}}\frac{\deg p}{|p|}\leq \sum_{\substack{ \beta < \deg p \\ \leq \log_q\log_qc + \delta -1}} \frac{\deg p}{|p|}+\sum_{\substack{\deg p \geq \log_q\log_q c + \delta \\ p\in \mathcal{P}}} \frac{\deg p}{q^\epsilon\log_q c} \\
&\leq \max\{\log_q\log_qc + \delta - 1 - \beta, 0\}+q^{-\epsilon} \\
&\leq \max \{\log_q\log_q||f||+[6,2,2]\log_qd - \beta +\epsilon, 0\} +\log_qO_\epsilon(1),
\end{align*}
where we have used that $k - \sum_{\deg p\leq k}\frac{\deg p}{|p|}=O(q^{-1/2})=\log_qO(1)$. 
\end{proof}

\begin{proposition}\label{thm: lower bound on determinant}
Let $\xi_1, ..., \xi_s$ be $\FF_q[t]$-points on $X$ and let $F_{li}\in \FF_q[t][x_0, ..., x_{n+1}], 1\leq l\leq L, 1\leq i\leq s$ be homogeneous polynomials. Let $\Delta_l$ be the determinant of the matrix $(F_{li}(\xi_j))_{ij}$. Let $\Delta$ be the greatest common divisor of the $\Delta_l$. If $\Delta$ is non-zero then for any $\epsilon>0$
\[
\deg\Delta \geq \frac{n!^{1/n}n}{n+1}s^{1+1/n}\left(\frac{1}{n}\log_qs - 1 - \epsilon - \beta - \log_q b(f) + \log_q O_{n,  \epsilon}(1)\right).
\]
\end{proposition}
\begin{proof}
Denote by $\mathcal{P}$ the collection of primes $p$ such that either $\deg p\leq \beta$ or $f\mod p$ is not absolutely irreducible. Fix $\delta\in [\epsilon, 1+\epsilon)$ such that $\frac{1}{n}\log_q s - \delta$ is an integer. Applying Lemma \ref{thm: estimate one prime} to all primes $p$ for which $|p|\leq s^{1/n}/q^\delta$ not in $\mathcal{P}$ yields that
\[
\deg \Delta \geq \frac{n!^{1/n}n}{n+1}s^{1+1/n}\sum_{\substack{|p|\leq q^{-\delta}s^{1/n} \\ p\not\in \mathcal{P}}}\frac{\deg p}{|p| + O_n(d^2|p|^{1/2})} - O_n(s)\sum_{|p|\leq q^{-\delta}s^{1/n}} \deg p.
\]
We first estimate the last sum. Since $\sum_{\deg p\leq k}\deg p=\frac{q}{q-1}q^k+O(q^{k/2})$ we have
\begin{align*}
\sum_{|p|\leq q^{-\delta}s^{1/n}}\deg p &= \frac{q}{q-1}s^{1/n}q^{-\delta} + O(\sqrt{s^{1/n}}q^{\frac{-1}{2}\delta}) \\
    &= s^{1/n}\log_qO_{n,\epsilon}(1).
\end{align*}
For the main term we can use that $\frac{1}{|p|+O(d^2|p|^{1/2})} \geq \frac{1}{|p|}-O_n(d^2)\frac{1}{|p|^{3/2}}$. Noting that $\sum_{\deg p\leq k}\frac{\deg p}{|p|} = k+\log_qO(1)$ we obtain
\begin{align*}
\sum_{\substack{|p|\leq q^{-\delta}s^{1/n} \\ p\notin \mathcal{P}}}& \frac{\deg p}{|p|+d^2|p|^{1/2}} \geq \sum_{|p|\leq q^{-\delta}s^{1/n}} \frac{\deg p}{|p|} - \sum_{p\in \mathcal{P}}\frac{\deg p}{|p|} - O_n(d^2)\sum_{\deg p > \beta }\frac{\deg p}{|p|^{3/2}} \\
    &\geq \frac{1}{n}\log_qs - \delta - \sum_{\deg p\leq \beta }\frac{\deg p}{|p|}-\log_qb(f)-\log_qO_n(1)\\
    &\geq \frac{1}{n}\log_qs - 1 - \epsilon - \log_qO_n(1) - \beta - \log_qb(f),
\end{align*}
where we have used that 
\begin{align*}
O_n(d^2)\sum_{\deg p>\beta}\frac{\deg p} {|p|^{3/2}} = d^2 O_n(q^{-\frac{1}{2}(\beta+1)})=O_n(q^{-1/2})=\log_qO_n(1).
\end{align*}
This proves the proposition.

\end{proof}

\paragraph{The Thue-Siegel lemma.} We have the following analogue of the improvement to the classical Thue-Siegel lemma, by Bombieri and Vaaler \cite{Bombi-Vaal}.

\begin{theorem}\label{thm: thue-siegel}
Let $A$ be an $s\times r$ matrix over $\FF_q[t]$ with $r>s$. Assume that $A$ has full rank $s$ and put $N=\max_{ij}\deg A_{ij}$. Then there is a non-zero $\FF_q[t]$-solution $x=(x_1, ..., x_r)$ to the system $AX=0$ with
\[
\max_{ij}\deg x_i\leq \frac{sN-\deg D}{r-s},
\]
where $D$ is the greatest common divisor of all $s\times s$ subminors of $A$.
\end{theorem}
\begin{proof}
Denote by $A_i$ the $i$-th component of the linear map $A: \FF_q[t]^r\to \FF_q[t]^s$. If $x=(x_1, ..., x_r)\in\FF_q[t]^r$ satisfies $\deg x_i< k$ for all $i$ and for some integer $k$ then $\deg A_i(x_1, ..., x_r)< N+k$. In particular we see that 
\[
A(\FF_q[t]_{<k}^r) \subseteq \{(y_1, ..., y_s)\mid \deg y_i< N+k\}.
\]
Transforming the lattice $\im A$ in Hermite normal form does not alter $\deg D$, and so we see that only a proportion of $|D|^{-1}$ elements of the right hand side are actually reached by $A$. Take $k=\left\lceil \frac{sN-\deg D}{r-s}\right\rceil$ and apply the pigeon hole principle to conclude.
\end{proof}

\section{Points of bounded height on curves}

\paragraph{The main estimate.} Fix a primitive irreducible polynomial $f$ in $\FF_q[t][x_0, ..., x_{n+1}]$. By the first paragraph of \cite[Sec\,4]{Walsh} we may assume that $f$ is absolutely irreducible. We make one further restriction. Denote by $c_f$ the coefficient of the monomial $x^d_0$ in $f$. We will assume that
\[
\deg c_f\geq \log_q||f|| - d\log_qd.
\]
Later we will show that any $f$ can be brought in such a form by an appropriate coordinate transformation.

\begin{theorem}\label{thm: main theorem}
Let $\ell$ be an integer. Then there exists a homogeneous polynomial $g$ not divisible by $f$, vanishing at all $\FF_q[t]$-points of $f$ of logarithmic height less than $\ell$, and of degree
\[
M \leq_{\epsilon,n} q^{1+\epsilon} q^{\frac{n+1}{nd^{1/n}}(\ell-1)} \frac{q^{\beta} d^{-1/n} b(f)} {||f||^{n^{-1}d^{-1-1/n}}} + q^{\epsilon}d^{1-1/n}(\ell-1) + q^{1+\epsilon}d^{\left[ 7, \frac{14}{3}-\frac{1}{n}, 3\right]}.
\]
\end{theorem}
\begin{proof}
Denote by $S$ the set of $\FF_q[t]$-points on $f=0$ of logarithmic height less than $\ell$. Let $M$ be an integer such that any homogeneous polynomial $g$ of degree $M$ vanishing on $S$ is divisible by $f$. We prove that $M$ is bounded as stated. Assume that $M\geq 4q^\epsilon d^2$.

For an integer $D$, denote by $\mathcal{B}(D)$ the set of monomials of degree $D$ in the variables $x_0, ..., x_{n+1}$. Note that $|\mathcal{B}(D)|=\binom{D+n+1}{n+1}$. Take a maximal algebraically independent set $T\subseteq S$ and put $s=|T|$. This means that the $s\times r$ matrix $A=(b(t))_{t\in T,b\in \mathcal{B}(M)}$ has full rank, where $r=|\mathcal{B}(M)|$. By assumption, any degree $M$ polynomial vanishing on $T$ is divisible by $f$. But these are precisely the elements of $f\cdot \mathcal{B}(M-d)$ and hence
\[
s=|\mathcal{B}(M)|-|\mathcal{B}(M-d)|.
\]
The solutions to the linear system described by $A$ correspond to degree $M$ homogeneous polynomials vanishing on $T$ (and so also on $S$). Since any such polynomial $g$ is divisible by $f$, $g$ has a coefficient of degree at least $\deg c_f\geq \log_q||f||-d\log_qd$. Since every entry of $A$ has degree at most $M(\ell-1)$, we get from Theorem \ref{thm: thue-siegel} that
\[
\deg \Delta \leq sM(\ell-1)-(r-s)(\log_q||f||-d\log_qd),
\]
where $\Delta$ is the greatest common divisor of the determinants of all $s\times s$ minors of $A$. Fix $\epsilon>0$. By Proposition \ref{thm: lower bound on determinant} we get that
\begin{align*}
\frac{n!^{1/n}n}{n+1}&\frac{s^{1/n}}{M}\left(\frac{1}{n}\log_q s - 1 - \epsilon - \beta - \log_qb(f) + \log_qO_{n, \epsilon}(1)\right) \\
&\leq (\ell-1)-\frac{r-s}{Ms}(\log_q||f||-d\log_qd).
\end{align*}
In the same way as in the proof of \cite[Lem.\,3.3.2 \& 3.3.3]{CCDN} we obtain that $\log_q s = \log_qd + n\log_qM + \log_qO_n(1)$ and that
\[
\frac{s^{1/n}}{M} = \frac{d^{1/n}}{n!^{1/n}} + O_n\left(\frac{d^2}{M}\right).
\]
Hence we may replace the left-hand side by
\begin{align*}
\frac{d^{1/n}n}{n+1}\left(1 + O_n\left(\frac{d^{2-1/n}}{M}\right)\right) \left(\log_qM -1-\epsilon - \beta + \frac{1}{n}\log_qd - \log_qb(f) + \log_qO_{n, \epsilon}(1)\right).
\end{align*}
By our assumption that $M\geq 4q^\epsilon d^2$ we get that $\frac{d^{2-1/n}\log_qM}{M}$ and $\frac{d^{2-1/n}\log_qd}{M}$ are both in $\log_qO_{n, \epsilon}(1)$. Thus we may further reduce the left-hand side to
\[
\frac{d^{1/n}n}{n+1}\left(\log_qM -1-\epsilon - \beta + \frac{1}{n}\log_qd - (1+O(d^{2-1/n}/M))\log_qb(f) + \log_qO_{n, \epsilon}(1)\right)
\]

To estimate the right-hand side, we get from the proof of \cite[Lem.\,3.3.2 \& 3.3.3]{CCDN} that
\[
\frac{r-s}{Ms}= \frac{1}{d(n+1)} + O_n(1/M).
\]
So the right-hand side becomes
\[
(\ell - 1) - \frac{\log_q||f||}{d(n+1)} + O_n\left(\frac{\log_q||f||}{M}\right) + \frac{\log_qd}{n+1} + O_n\left( \frac{d\log_qd}{M}\right).
\]
The last two terms may be absorbed in the $\log_q O_{n, \epsilon}(1)$ on the left-hand side. Putting this all together we obtain that
\begin{align*}
\frac{d^{1/n}n}{n+1}&\left(\log_qM -1-\epsilon - \beta + \frac{1}{n}\log_qd - (1+O(d^{2-1/n}/M))\log_qb(f) + \log_qO_{n, \epsilon}(1)\right) \\
&\leq (\ell - 1) - \frac{\log_q||f||}{d(n+1)} + O_n\left(\frac{\log_q||f||}{M}\right).
\end{align*}

We now treat the two cases were $||f||$ is either small or large. Assume first that $M\geq d^{1-1/n}(\ell-1) q^\epsilon$ and that $\log_q||f||\leq 4(n+1)d(\ell-1)$. Then also $\log_q||f||\leq d^{1/n}M\log_q O_{n, \epsilon}(1)$ and so we may further absorb the term $\log_q||f||/M$ into the $\log_qO_{n, \epsilon}(1)$ on the left-hand side. For the same reason, we may also remove the term $O_n(d^{2-1/n}/M)\log_qb(f)$, since it is in $\log_q O_{n, \epsilon}(1)$. Rearranging yields that
\[
\log_qM \leq 1+\epsilon+\log_qO_{n,\epsilon}(1) + \frac{n+1}{d^{1/n}n}(\ell-1) + \beta - \frac{1}{n}\log_qd + \log_qb(f) - \frac{\log_q||f||}{nd^{1+1/n}},
\]
as desired. 

Secondly assume that $\log_q||f||> 4(n+1)d(\ell-1)$. Then $\frac{n+1}{d^{1/n}n}(\ell-1) \leq \frac{\log_q ||f||}{4d^{1+1/n}}$. We also assume that $M\geq O_n(1)d$ where the implicit constant is chosen such that the term $O_n\left(\frac{\log_q||f||}{M}\right)$ is bounded by $\frac{\log_q ||f||}{4nd}$. Rearranging shows that
\[
\log_qM \leq 1+\epsilon+\log_qO_{n,\epsilon}(1) +\beta - \frac{1}{n}\log_qd + \left(1+O_n\left(\frac{d^{2-1/n}}{M}\right)\right)\log_qb(f) - \frac{\log_q||f||}{2nd^{1+1/n}}.
\]
Putting $c=\left(1+O_n\left(\frac{d^{2-1/n}}{M}\right)\right)$, we have that
\begin{align*}
&c\log_qb(f) - \frac{\log_q||f||}{2nd^{1+1/n}} \\
 &\leq \max\{c\epsilon + c[6,2,2]\log_qd - c\beta + c\log_q\log_q||f|| - \frac{\log_q||f||}{2nd^{1+1/n}}, 0\} + \log_qO_{n, \epsilon}(1).
\end{align*}
Now note that $\log_q\log_qx - \frac{\log_qx}{c}\leq \log_qc+\log_qO(1)$ and so
\begin{align*}
c\left(\log_q\log_q||f|| - \frac{\log_q||f||}{2cnd^{1+1/n}}\right) 
    \leq c\log_qc + c\log_qd^{1+1/n} + c\log_qO_n(1).
\end{align*}
But since $c=\left(1+O_n\left(\frac{d^{2-1/n}}{M}\right)\right)$ we obtain that this quantity is $\log_qd^{1+1/n} + \log_qO_{n, \epsilon}(1)$. Therefore
\[
c\log_qb(f) - \frac{\log_q||f||}{2nd^{1+1/n}} \leq \left([7,3,3]+\frac{1}{n}\right)\log_qd + \epsilon - \beta + \log_qO_{n, \epsilon}(1).
\]
It follows that 
\[
\log_q M \leq 1+\epsilon+ \log_qO_{n,\epsilon}(1) + \left[ 7, \frac{14}{3} - \frac{1}{n}, 3 \right]\log_qd.\qedhere
\]
\end{proof}

With this, we can prove our main result for projective curves.

\begin{proof}[Proof of Theorem \ref{thm: projective counting}]
Apply Theorem \ref{thm: main theorem} for $n=1$ to obtain a polynomial $g$ of degree at most
\[
\leq_\epsilon q^{1+\epsilon}q^{\frac{2(\ell-1)}{d}}\frac{q^{\beta-\log_q d}b(f)}{||f||^{1/d^2}}+q^\epsilon(\ell-1) + q^{1+\epsilon}d^{[7, 11/3, 3]}.
\]
vanishing at all $\FF_q[t]$-points on $f=0$ of logarithmic height less than $\ell$. Apply Lemma \ref{thm: bound b(f)} and use the fact that $\frac{\log_q||f||^{1/d^2}}{||f||^{1/d^2}}=O(q^\epsilon)$ to obtain that
\[
\deg g\leq_{\epsilon}q^{1+\epsilon}q^{\frac{2(\ell-1)}{d}}d^{[7, 11/3, 3]}.
\]
Apply B\'ezout's theorem to conclude.
\end{proof}

\paragraph{The coordinate transformation.} We finish up the argument above with the required coordinate transformation.

\begin{lemma}
Let $f\in \FF_q[t][x]$ be of degree $\leq d$. Then there exists a polynomial $\alpha\in \FF_q[t]$ with $|\alpha|\leq d$ and such that $|f(\alpha)|\geq ||f||$. 
\end{lemma}
\begin{proof}
Take $d+1$ disctint polynomials $\alpha_0, ..., \alpha_d$ in $\FF_q[t]$ of smallest degree possible. Letting $\kappa=\lfloor \log_q d \rfloor$ this implies that $\deg \alpha_i\leq \kappa$. Now denote by $V$ the $(d+1)\times (d+1)$ Vandermonde matrix on the $\alpha_i$. Denote by $M_f$ the vector $(f(\alpha_i))_{i=0, ..., d}$ and by $C_f$ the coefficient vector of $f$, then $VC_f=M_f$ and so
\[
|c_f| = ||C_f||_\infty \leq ||V^{-1}||_\infty||M_f||_\infty.
\]
The inverse of $V$ has entries given by 
\[
V_{kj}^{-1} = \frac{1}{\displaystyle\prod_{m\neq j}(\alpha_j-\alpha_m)}\begin{cases}
\pm \displaystyle\sum_{\substack{0\leq m_1<...<m_{d-k}\leq d \\ m_i\neq j}}\alpha_{m_1}\cdots \alpha_{m_{d-k}} & 0\leq k< d \\
\pm 1 & k=d.
\end{cases}
\]
This explicit description shows that $|V_{kj}^{-1}|\leq 1$ and the result follows.
\end{proof}

\begin{lemma}\label{thm: coordinate transformation}
Let $f\in \FF_q[t][x_0, ..., x_{n+1}]$ be homogeneous of degree $d$. Then there exist polynomials $\alpha_1, ..., \alpha_{n+1}$ with $|\alpha_i|\leq d$ and such that 
\[
|f(1,\alpha_1, ..., \alpha_{n+1})|\geq ||f||.
\]
\end{lemma}
\begin{proof}
Put $x_0=1$ and use induction with the previous lemma.
\end{proof}

Let us now check that Theorem \ref{thm: main theorem} still holds for any $f$. So let $f$ be an arbitrary primitive absolutely irreducible polynomial in $\FF_q[t][x_0, ..., x_{n+1}]$. Take $\alpha_1, ..., \alpha_{n+1}$ in $\FF_q[t]$ such that $|\alpha_i|\leq d$ and $|f(1,\alpha_1, ..., \alpha_{n+1})|\geq ||f||$. Define 
\[
f'(x_0, ..., x_{n+1}) = f(x_0, x_1+\alpha_1x_0 ..., x_n+\alpha_nx_0, x_{n+1}+\alpha_{n+1}x_0).
\]
By definition we have that $|c_{f'}|\geq ||f||$. But since $|\alpha_i|\leq d$ we also have that $||f'||\leq ||f||d^d$. Therefore $\deg c_f \geq \log_q||f'|| - d\log_qd$. By Theorem \ref{thm: main theorem} there is a polynomial $g'$ of degree
\[
\leq_{\epsilon,n} q^{1+\epsilon} q^{\frac{n+1}{nd^{1/n}}(\ell-1)} \frac{q^{\beta - \frac{1}{n}\log_q d}b(f)} {||f||^{n^{-1}d^{-1-1/n}}} + q^{1+\epsilon}d^{1-1/n}(\ell-1) + q^{1+\epsilon}d^{\left[ 7, \frac{14}{3}-\frac{1}{n}, 3\right]}.
\]
vanishing on all points of $f'=0$ logarithmic height less than $\ell+\lfloor \log_q d \rfloor$. Here we have used that $q^{(n+1)\log_qd/ d^{1/n}}=O_n(1)$. Then the polynomial 
\[
g(x_0, ..., x_{n+1}) = g'(x_0, x_1-\alpha_1x_0 ..., x_n-\alpha_n, x_{n+1}-\alpha_{n+1}x_0)
\]
has the same degree as $g'$ and vanishes on all points of $f=0$ of logarithmic height less than $\ell$, proving Theorem \ref{thm: main theorem} for $f$. 

\paragraph{Affine curves.} For counting on affine curves we use a technique from \cite{Ellenb-Venkatesh}, refined in \cite[Sec.\,4.2]{CCDN}.

\begin{lemma}\label{thm: elementary bound affine}
Let $F\in \FF_q[t][x_0, ..., x_{n+1}]$ be primitive homogeneous of degree $d$. For $1\leq y \leq ||F||$ we have 
\[
q^{\beta}d^{-1/n}\frac{b(F)}{||F||^{\frac{1}{nd^{1+1/n}}}} \leq_{\epsilon, n} q^\epsilon \frac{d^{[6,2,2]-1/n}\log_qy + d^{[7, \frac{14}{3}-\frac{1}{n}, 3]}} {y^{\frac{1}{nd^{1+1/n}}}}.
\]
\end{lemma}
\begin{proof}
If $||F||=1$ then $y=1$ and we are done. So assume that $||F||>1$. The map
\[
x\mapsto \frac{\log_q x}{x^{\frac{1}{nd^{1+1/n}}}}
\]
is increasing on $(1, c)$ and decreasing on $(c, \infty)$, with $c=e^{nd^{1+1/n}}$. It has a global maximum at $x=c$ with value $\frac{nd^{1+1/n}\log_qe}{e}=\log_qO_n(1)\cdot d^{1+1/n}$. By Lemma \ref{thm: bound b(f)} we have $b(F)\leq_\epsilon q^\epsilon (q^{-\beta}d^{[6,2,2]}\log_q ||F|| +1)$. Assume first that $y\geq e^{nd^{1+1/n}}$, then
\begin{align*}
q^{\beta}d^{-1/n}\frac{b(F)}{||F||^{\frac{1}{nd^{1+1/n}}}} &\leq_{\epsilon, n} q^\epsilon d^{-1/n} \frac{d^{[6,2,2]}\log_q||F|| + q^\beta} {||F||^{\frac{1}{nd^{1+1/n}}}} \\
&\leq_{\epsilon, n} q^\epsilon \frac{d^{[6,2,2]-1/n} \log_qy + d^{[\frac{14}{3}, \frac{14}{3}, 0]-1/n}}{y^{\frac{1}{nd^{1+1/n}}}}.
\end{align*}
Now suppose that $y\leq e^{nd^{1+1/n}}$. Then
\begin{align*}
q^{\beta}d^{-1/n}\frac{b(F)}{||F||^{\frac{1}{nd^{1+1/n}}}} &\leq_{\epsilon, n} q^\epsilon d^{\frac{-1}{n}} \left( \frac{d^{[6,2,2]\log_q e^{nd^{1+1/n}}}}{e} + \frac{q^\beta}{y^{\frac{1}{nd^{1+1/n}}}}\right) \\
&\leq_{\epsilon, n} q^\epsilon \frac{d^{[7,3,3]} + d^{[\frac{14}{3}, \frac{14}{3}, 0]-\frac{1}{n}}}{y^{\frac{1}{nd^{1+1/n}}}} \leq_{\epsilon, n} q^\epsilon \frac{d^{[7, \frac{14}{3}-\frac{1}{n}, 3]}}{y^{\frac{1}{nd^{1+1/n}}}}. \qedhere
\end{align*}
\end{proof}

\begin{lemma}\label{thm: affine large coefficients}
Let $f\in\FF_q[t][x_1, ..., x_{n+1}]$ be a primitive irreducible polynomial of degree $d\geq 2$. Then either $\log_q||f||\leq \ell d\binom{d+n+1}{n+1}$ or there exists a non-zero degree $d$ polynomial $g\in \FF_q[t][x_1, ..., x_{n+1}]$ coprime to $f$ and vanishing on all points of $f=0$ of logarithmic height less than $\ell$.
\end{lemma}
\begin{proof}
Put $\theta=\binom{d+n+1}{n+1}$, this is simply the number of monomials of degree $\leq d$ in $x_1, ..., x_{n+1}$. Let $p_1, ..., p_N$ be the points in $\AA^{n+1}$ on $V(f)$ with $h(p_i)<\ell$ and let $A$ be the matrix whose rows are the $\theta$ monomials of degree $\leq d$ evaluated at the $p_i$. The solutions to $AX=0$ describe polynomials of degree $\leq d$ vanishing on $p_1, ..., p_N$. Since $f$ vanishes there, the rank of $A$ is at most $\theta-1$. Hence we may construct a non-zero solution to $AX=0$ whose entries are subdeterminants of $A$. In particular, $\deg X_i\leq \ell d\theta$. If $g$ is the polynomial corresponding to $X$ then either $f\mid g$ and so $\log_q||f||\leq \ell d\theta$, or $g$ and $f$ are coprime as stated.
\end{proof}

\begin{lemma}\label{thm: main lemma affine counting}
Let $f\in \FF_q[t][x_1, ..., x_{n+1}]$ be primitive and irreducible of degree $d\geq 2$. Denote by $f_i$ the degree $i$ part of $f$. Fix a positive integer $\ell$. Then there exists a $g\in\FF_q[t][x_1, ..., x_n]$ not divisible by $f$ vanishing on all points of $f$ of logarithmic height less than $l$ and of degree 
\begin{align*}
M\leq_{\epsilon, n}\, &q^{1+\epsilon}q^{\frac{\ell-1}{d^{1/n}}}  \frac{\min\{d^{[6,2,2]-1/n}(\log_q ||f_d|| + d(\ell-1)) +d^{[7,\frac{14}{3} -\frac{1}{n}, 3]}, q^{\beta}d^{- 1/n} b(f)\}}{||f_d||^{\frac{1}{nd^{1+1/n}}}}\\
&\qquad + q^{\epsilon}d^{1-1/n}(\ell-1) + q^{1+\epsilon}d^{\left[ 7, \frac{14}{3}-\frac{1}{n}, 3\right]}.
\end{align*}
\end{lemma}
\begin{proof}
By the same reasoning as in \cite[Sec.\,4]{Walsh} applied to the homogenization of $f$ we may assume that $f$ is absolutely irreducible. For a monic polynomial $H\in \FF_q[t]$ define the degree $d$ homogeneous polynomial $F_H(x_0, ..., x_n) = \sum_{i=0}^d H^if_ix_0^{d-i}$. Note that every $\FF_q[t]$-point $(x_1, ..., x_n)$ of $f=0$ gives us the point $(H: x_1: ...: x_n)$ on $F_H=0$. 

Suppose first that $q^{\ell-1} \leq_n d^{O_n(1)}$. Then $q^{\frac{\ell-1}{nd^{1/n}}}=O_n(1)$ and so we use Theorem \ref{thm: main theorem} to obtain a homogeneous polynomial $G_1$ not divisible by $F_1$ of degree 
\[
\leq_{\epsilon, n} q^{1+\epsilon}q^{\frac{\ell-1}{d^{1/n}}}\frac{q^{\beta}d^{-1/n} b(F_1)}{||F_1||^{\frac{1}{nd^{1+1/n}}}} + q^{\epsilon}d^{1-1/n}(\ell-1) + q^{1+\epsilon}d^{\left[ 7, \frac{14}{3}-\frac{1}{n}, 3\right]},
\]
vanishing on all points of $F_1$ of logarithmic height less than $\ell$. As $b(F_1)=b(f)$ and $||F_1||\geq ||f_d||$ we obtain via Lemma \ref{thm: elementary bound affine} that $\deg G_1$ is bounded by 
\begin{align*}
\leq_{\epsilon, n} & q^{1+\epsilon}q^{\frac{\ell-1}{d^{1/n}}} \frac{\min\{d^{[6,2,2]-\frac{1}{n}}\log_q ||f_d|| + d^{[7, \frac{14}{3}-\frac{1}{n}, 3]}, q^{\beta}d^{-1/n} b(f)\}} {||f_d||^{\frac{1}{nd^{1+1/n}}}}  \\
&\qquad + q^{\epsilon}d^{1-1/n}(\ell-1) + q^{1+\epsilon}d^{\left[ 7, \frac{14}{3}-\frac{1}{n}, 3\right]}.
\end{align*}
Dehomogenizing $G_1$ gives the desired bound.

Now suppose that there is a prime $p$ of degree $\ell-1$ such that $p\nmid f_0$. Then $F_p$ is primitive and so by Theorem \ref{thm: main theorem} there exists a homogeneous $G_p$ not divisible by $F_p$ of degree 
\[
\leq_{\epsilon, n} q^{1+\epsilon}q^{\frac{n+1}{nd^{1/n}}(\ell-1)} \frac{q^{\beta}d^{-1/n} b(F_p)}{||F_p||^{\frac{1}{nd^{1+1/n}}}} + q^{\epsilon}d^{1-1/n}(\ell-1) + q^{1+\epsilon}d^{\left[ 7, \frac{14}{3}-\frac{1}{n}, 3\right]},
\]
vanishing on all points of logarithmic height less than $\ell$ on $F_p=0$. Moreover, since $||F_p||\geq |p|^d||f_d|| = q^{d(\ell-1)}||f_d||$ and since $b(F_p)$ and $b(f)$ agree up to $q^{\deg p/|p|}=O(1)$, we see by applying Lemma \ref{thm: elementary bound affine} that
\begin{align*}
\deg G_p &\leq_{\epsilon, n} q^{1+\epsilon} q^{\frac{\ell-1}{d^{1/n}}} \frac{\min\{d^{[6,2,2]-\frac{1}{n}}(\log_q||f_d|| + d(\ell-1)) + d^{[7, \frac{14}{3}-\frac{1}{n}, 3]}, q^{\beta}d^{-1/n}b(f) \}} {||f_d||^{\frac{1}{nd^{1+1/n}}}} \\
& \qquad + q^{\epsilon}d^{1-1/n}(\ell-1) + q^{1+\epsilon}d^{\left[ 7, \frac{14}{3}-\frac{1}{n}, 3\right]}.
\end{align*}
Taking the dehomogization of $G_p$ gives the bound. 

Now assume that $p\mid f_0$ for any prime $p$ of degree $\ell-1$. If $f_0$ is non-zero then 
\[
\sum_{\deg p = \ell-1} \deg p\leq \log_q |f_0|.
\]
The left-hand side is roughly of size $q^{\ell-1} + O(q^{(\ell-1)/2})$ by the prime number theorem. If $\log_q |f_0|\geq \ell d\binom{d+n+1}{n+1}$ then we apply Lemma \ref{thm: affine large coefficients} to $F_1$ and we're done. So we have $\log_q |f_0|\leq \ell d\binom{d+n+1}{n+1}$. Then 
\[
\sum_{\deg p = \ell-1} \deg p = q^{\ell-1} + O(q^{(\ell-1)/2})\leq d\ell \binom{d+n+1}{n+1},
\]
and so $q^{\ell-1-\epsilon} \leq_{\epsilon, n} d^{O_n(1)}$. Then we are done by the discussion above.

Finally, if $f_0=0$, then by Lemma \ref{thm: coordinate transformation} (applied to the homogenization of $f$) we find $\alpha_1, ..., \alpha_n$ with $f(\alpha_1, ..., \alpha_n)\neq 0$ and $|\alpha_i|\leq d$. We then apply the above discussion to $g(x_1, ..., x_n) = f(x_1+\alpha_1, ... x_n+\alpha_n)$ with $\ell' = \ell+ \lfloor\log_q d\rfloor$ to finish the proof.
\end{proof}

We can prove the affine counting result. 

\begin{proof}[Proof of Theorem \ref{thm: affine counting}]
By the previous result, there is a polynomial $g$ not divisible by $f$ vanishing at all $\FF_q[t]$-points of $f=0$ of logarithmic height less than $\ell$ and of degree
\[
\leq_\epsilon q^{1+\epsilon} q^{\frac{\ell-1}{d}}\frac{d^{[5,1,1]}(\log_q ||f_d||+d(\ell-1)) + d^{[7, \frac{11}{3}, 3]}}{||f_d||^{\frac{1}{d^2}}} + q^{\epsilon}(\ell-1) + q^{1+\epsilon}d^{[7, \frac{11}{3}, 3]}.
\]
Use the fact that $\frac{\log_q ||f_d||}{||f_d||^{1/d^2}} = O(q^\epsilon d^2)$ to get 
\[
\deg g\leq_\epsilon q^{1+\epsilon}q^{\frac{\ell-1}{d}} d^{[6,2,2]} (d^{[1, \frac{5}{3}, 1]}+\ell).
\]
Apply B\'ezout to conclude.
\end{proof}

\section{Dimension growth for hypersurfaces}

In this section we prove our dimension growth result. We need a projection lemma on curves in $\AA^3$, which will be proven later.

\begin{lemma}\label{thm: projecting curves in A3}
Let $C$ be a curve in $\AA^3$ of degree $d$. Then there exists a curve $C'$ in $\AA^2$ birational to $C$ and of the same degree such that
\[
N(C; \ell) \leq  N(C'; \ell + 2\log_qd)+d^2.
\]
\end{lemma}

The main goal of this section is to prove the following version of the dimension growth conjecture for affine hypersurfaces, which will immediately imply Theorem \ref{thm: projective dimension growth}.

\begin{theorem}\label{thm: dimension growth affine hypersurfaces}
Let $n\geq 3$ and let $f\in \FF_q[t][x_1, ..., x_n]$ be a polynomial of degree $d\geq 64$ such that the degree $d$ part $f_d$ of $f$ is absolutely irreducible. Then
\[
N_\mathrm{aff}(f; \ell)\leq_{\epsilon, n}q^{9 +\epsilon}d^{O_n(1)} q^{\ell(n-2)}.
\]
If moreover, $\ell\geq 10$ satisfies $\frac{\ell-1}{\ell-9}\leq \frac{\sqrt{d}}{8}$ then
\[
N_\mathrm{aff}(f; \ell)\leq_{\epsilon, n} q^{1+\epsilon}d^{O_n(1)} q^{\ell(n-2)}.
\]
\end{theorem}

\paragraph{Dimension growth for surfaces.} We prove dimension growth by induction on the dimension of our hypersurface, with the base case being surfaces.

\begin{lemma}
Let $f\in \FF_q[t][x_1, x_2, x_3]$ be a degree $d\geq 3$ polynomial whose highest degree part $f_d$ is irreducible and let $I$ be a set of lines in $\AA^3$ lying on the surface $X$ defined by $f$. Then, we have
\[
N(X\cap (\cup_{L\in I}L); l)\leq_{\epsilon} q^{1+\epsilon}d^{O(1)} q^{\ell} + \#I.
\]
\end{lemma}
\begin{proof}
Define $I_1=\{L\in I\mid N_\mathrm{aff}(L; l)\leq 1\}$ and $I_2=\{L\in I\mid N_\mathrm{aff}(L; l)\geq 2\}$. If $L\in I_2$ then there exist $a,v\in \FF_q[t]^3$ such that $h(a)<\ell$, $v$ is primitive and $L=a+v\FF_q(t)$. Since $L\in I_2$ we see that $h(v)<\ell$ and hence
\[
N_\mathrm{aff}(L; l) =q^{\ell-h(v)}.
\]
Defining the polynomial $g(x)=f(x+a)$ we obtain for any $\lambda\in \FF_q(t)$
\[
0=f(\lambda v+a) = g(\lambda v) = \sum_{i=0}^d \lambda^ig_i(v).
\]
Therefore $f_d(v)=g_d(v)=0$. Note that the line $L$ intersects a generic plane $H$ in $\AA^3$ in a point $b$ which is a zero of both $f$ and the directional derivative $D_vf$. In particular, for a primitive $v\in \FF_q[t]^3$, the number of $L\in I_2$ with direction $v$ is at most $d(d-1)$ by B\'ezout's theorem. Define
\[
A_i = \{v\in \PP^2(\FF_q(t)) \mid f_d(v) = 0, h(v)=i\}, n_i=\#A_i.
\]
Then by Theorem \ref{thm: projective counting}
\[
\sum_{i=0}^{k-1}n_i = N(f_d; k)\leq_{\epsilon}q^{1+\epsilon}d^{8} q^{\frac{2(k-1)}{d}}.
\]
Thus $N(X\cap (\cup_{L\in I}; \ell))\leq \#I_1 + d(d-1)\sum_{i=0}^{l-1}n_iq^{\ell-i}$. To estimate this sum, apply partial summation to get
\begin{align*}
\sum_{i=0}^{\ell-1}n_iq^{l-i} &\leq q\sum_{i=0}^{\ell-1}n_i + \sum_{j=0}^{\ell-2}(q^{\ell-j} - q^{\ell-j-1})\sum_{i=0}^j n_i \\
&\leq_{\epsilon} q^{2+\epsilon}d^{8} q^{\frac{2(\ell-1)}{d}} + q^{2+\epsilon}d^{8} q^{\ell-1} \frac{q-1}{q}\sum_{j=0}^{\ell-2}\frac{1}{q^{j\left(1-\frac{2}{d}\right)}}.
\end{align*}
As $d\geq 3$, we have $\sum_{j=0}^{\ell-2}(1/q^{1-2/d})^j\leq \sum_{j=0}^{\infty}(1/q^{1-2/d})^j=1+O(q^\epsilon)$. Putting this together gives
\[
N(X\cap(\cup_{L\in I}L); \ell)\leq_{\epsilon} q^{1+\epsilon}d^{10} q^{\ell} + \# I. \qedhere
\]
\end{proof}

\begin{proposition}\label{thm: dimension growth for affine surfaces} Let $f\in \FF_q[t][x_1, x_2, x_3]$ be of degree $d\geq 64$ and suppose that $f_d$ is absolutely irreducible. Let $X$ be the surface in $\AA^3$ defined by $f$. Then
\[
N_\mathrm{aff}(X; \ell)\leq_{\epsilon} q^{9 +\epsilon} d^{O(1)} q^{\ell}.
\]
Moreover, for $\ell\geq 10$ satisfying $\frac{\ell-1}{\ell-9}\leq \frac{\sqrt{d}}{8}$ we have
\[
N_\mathrm{aff}(X; \ell)\leq_{\epsilon} q^{1+\epsilon}d^{O(1)}q^{\ell}.
\]
\end{proposition}
\begin{remark}
If one fixes a $d\geq 65$ then this condition above is automatically satisfied for all large $\ell$. Conversely, if we fix $\ell\geq 10$ then this condition is also automatically true for large $d$. 
\end{remark}
\begin{proof}
Since $f_d$ is absolutely irreducible we have $b(f)\leq b(f_d)$. Applying Lemma \ref{thm: bound b(f)} and \ref{thm: main lemma affine counting} gives a polynomial $g\in \FF_q[t][x_1, x_2, x_3]$ vanishing on all points of $X$ of logarithmic height at most $\ell$ and of degree $\leq_\epsilon q^{1+\epsilon} d^7 q^{\frac{\ell-1}{\sqrt{d}}}$. Let $\mathcal{C}$ be the (reduced) intersection of $X$ and $g=0$. Let $L_1, ..., L_k$ be the degree $1$ components of $\mathcal{C}$ and let $C_1, ..., C_m, C_1', ..., C_n'$ be the irreducible components of $\mathcal{C}$ with 
\begin{align*}
1<&\deg C_i\leq \ell, \quad  \text{ for all }i, \\
 \ell<&\deg C_i',\phantom{\leq \ell,} \quad \text{ for all }i.
\end{align*}
Note that
\[
k + \sum_i\deg C_i + \sum_i\deg C_i' \leq d\deg g.
\]
The total contribution to $N_\mathrm{aff}(X; \ell)$ coming from the lines $L_i$ is at most $\leq_\epsilon q^{1+\epsilon}d^{O(1)} q^{\ell}$. For the $C_i$ we have by Theorem \ref{thm: affine counting}
\[
N_\mathrm{aff}(C_i; \ell)\leq_\epsilon q^{1+\epsilon} \deg (C_i)^8 \ell q^{\frac{\ell-1}{\deg C_i}}.
\]
If $\deg C_i=2$ then this quantity is clearly $\leq_\epsilon q^{1+\epsilon}q^{\frac{\ell-1}{2}}\ell$. If $\deg C_i> 2$ then using the fact that $\ell^7q^{\frac{-\ell}{6}}=O(1)$ also gives this bound. Thus
\[
N_\mathrm{aff}(C_i; \ell)\leq_\epsilon q^{1+\epsilon}q^\frac{\ell-1}{2}\ell.
\]
Therefore the total contribution of the $C_i$ to $N_\mathrm{aff}(X; \ell)$ is certainly bounded as stated.

Lastly, to bound the contribution from the high degree components we have
\begin{align*}
\sum_i N_\mathrm{aff}(C_i'; \ell) &\leq_\epsilon q^{1+\epsilon}\sum_i \deg ( C_i')^{7}q^{\frac{\ell-1}{\deg C_i'}}(\ell+\deg C_i') \\
    &\leq_\epsilon q^{2+\epsilon}\sum_i \deg (C_i)'^{8} \leq_\epsilon  q^{2+\epsilon}\left(\sum_i \deg C_i'\right)^{8} \\
    &\leq_\epsilon q^{9 +\epsilon}d^{O(1)} q^{8\frac{\ell-1}{\sqrt{d}}}.
\end{align*}
This proves the first statement for $d\geq 64$. If we assume that $\ell\geq 10$ and $\frac{\ell-1}{\ell-9}\leq \frac{\sqrt{d}}{8}$ then 
\[
q^{9 +\epsilon}d^{O(1)}q^{8\frac{\ell-1}{\sqrt{d}}} \leq q^{1+\epsilon}d^{O(1)}q^{\ell},
\]
proving the second statement as well.
\end{proof}

\paragraph{The general case.} For the induction procedure, we consider shifts of a certain hyperplane to reduce to varieties of lower dimension. We first prove that we can find a sufficiently good hyperplane.

\begin{lemma}\label{thm: existence good hyerplane}
Let $n\geq 3$ and $f\in \FF_q[t][x_0, ..., x_n]$ be absolutely irreducible of degree $d\geq 2$ defining a hypersurface $X$. Then there exists a non-zero homogeneous $F\in \FF_q[t][y_0, ..., y_n]$ of degree at most $12(n+1)d^7$ such that for a point $A=(a_0:...:a_n)\in (\PP^n)^*$, we have that if $H_A\cap X$ is not geometrically integral, then $F(A)=0$. Here $H_A$ is the hyperplane in $(\PP^n)^*$ associated to $A$.
\end{lemma}
\begin{proof}
Let $A=(a_0:...:a_n)\in (\PP^n)^*$ and assume that $a_0=1$. Then $H_A\cap X$ is geometrically integral if and only if $f(-a_1x_1-...-a_nx_n, x_1, ..., x_n)$ is irreducible over $\overline{\FF_q(t)}$. By \cite[Thm.\,7]{KALTOFEN1995} there exists a Noether form $\Phi$ of degree at most $12d^6$ such that if $H_A\cap X$ is not geometrically integral then $\Phi$ applied to the coefficients of $f(-a_1x_1-...-a_nx_n, x_1, ..., x_n)$ is $0$. Let $F_0\in \FF_q[t][y_0, y_1, ..., y_n]$ be the homogenization of 
\[
\Phi(\operatorname{coeff} (f(-y_1x_1-...-y_nx_n, x_1, ..., x_n))),
\]
where $\operatorname{coeff}$ maps a polynomial in the $x_i$ to its vector of coefficients. Note that $\deg F_0\leq 12d^7$. In the same manner, we find $F_1, ..., F_n$ such that for $A=(a_0: ...: a_n)$ with $a_i\neq 0$ we have that if $H_A\cap X$ is not geometrically integral then $F_i(A)=0$. We define $F=\prod_i F_i$. Because $\dim X\geq 2$ and $X$ is geometrically integral, \cite[Thm.\,6.10]{bertinithm} implies that the generic hyperplane $H$ in $\PP^n$ defines a geometrically integral intersection $H\cap X$. Hence there is a hyperplane $H_A$ at which none of the $F_i$ vanish and so $F$ is non-zero.
\end{proof}

\begin{proof}[Proof of Theorem \ref{thm: dimension growth affine hypersurfaces}]
Let $X\subseteq \AA^n$ be a geometrically integral hypersurface described by a polynomial $f$ of degree $d\geq 64$ with $f_d$ absolutely irreducible. We induct on $n$, where the base case $n=3$ follows from Proposition \ref{thm: dimension growth for affine surfaces}. So let $n>3$. 

The polynomial $f_d$ defines a geometrically integral hypersurface in $\PP^{n-1}$, so by Lemma \ref{thm: existence good hyerplane} there exists a non-zero homogeneous $F\in \FF_q[t][y_1,..., y_n]$ such that if $F(A)\neq 0$ then $H_A\cap V(f_d)$ is geometrically integral. Applying Lemma \ref{thm: coordinate transformation} to $F$ gives a point $A=(a_1: ...: a_n)\in (\PP^{n-1})^*$ with $a_i\in \FF_q[t]$, $|a_i|\leq 12nd^7$ and $H_A\cap V(f)$ geometrically integral.

Now note that if $f(x_1, ..., x_n)=0$ with $\deg x_i<\ell$ then
\[
\deg\left(\sum_i a_ix_i\right) < \log_q(12nd^7)+\ell,
\]
and hence
\[
N(f; \ell)\leq \sum_{\deg \alpha < \ell+\log_q(12nd^7)} N\left(V\left(f, \sum_i a_ix_i=\alpha\right); \ell\right).
\]
The set $V(f)\cap \left(\sum_i a_ix_i=\alpha\right)$ defines a hypersurface in $\sum_i a_ix_i=\alpha$ and after a coordinate transformation is described by a polynomial $g\in \FF_q[t][x_1, ..., x_{n-1}]$ of degree $d$ with $g_d$ absolutely irreducible. The induction hypothesis proves the result.
\end{proof}

Finally, we treat dimension growth for projective hypersurfaces.

\begin{proof}[Proof of Theorem \ref{thm: projective dimension growth}]
Let $n\geq 3$ and let $X$ be an irreducible projective hypersurface in $\PP^n$ of degree $d\geq 64$ defined by an irreducible $f\in \FF_q[t][x_0, ..., x_n]$. If $f$ is absolutely irreducible then $f$ also defines an affine hypersurface in $\AA^{n+1}$ and we have
\[
N(f; \ell)\leq N_{\mathrm{aff}}(f; \ell).
\]
Then the dimension growth for affine hypersurfaces gives the result.

Suppose now that $f$ is not absolutely irreducible. By \cite[Sec.\,4]{Walsh} there is then a homogeneous polynomial $g\in \FF_q[t][x_0, ..., x_n]$ of degree $d$ coprime to $f$ and vanishing on all $\FF_q(t)$-points of $X$. Thus the $\FF_q[t]$-points of $X$ are contained in the $(n-2)$-dimensional variety $X\cap V(g)$ of degree $d^2$. By Lemma \ref{thm: trivial bound affine} below we get
\[
N(f; \ell)\leq N_\mathrm{aff}(X\cap V(g); \ell) \leq d^2q^{\ell(n-1)}. \qedhere
\]
\end{proof}

\begin{lemma}\label{thm: trivial bound affine}
Let $Y$ be a (reduced) variety in $\AA^n_{\overline{\FF_q(t)}}$ of dimension $m$ and degree $d$, defined over $\overline{\FF_q(t)}$. Then for any $\ell$
\[
N(Y; \ell)\leq dq^{\ell m}.
\]
\end{lemma}
\begin{proof}
We induct on $m$, the case $m=0$ being trivial. So let $m>0$ and let $Y_1, ..., Y_k$ be the irreducible components of $Y$. For every $Y_i$ there is a coordinate $x_{j(i)}$ such that $\dim(Y_i\cap V(x_{j(i)}-a))<\dim Y_i$. Hence
\begin{align*}
N(Y; \ell) \leq \sum_i N(Y_i; \ell) \leq \sum_i \sum_{\deg a< \ell} N\left( Y_i\cap V(x_{j(i)}-a); \ell\right) \leq dq^{\ell m},
\end{align*}
by induction.
\end{proof}

\paragraph{Projecting curves in $\AA^3$.} To conclude, we prove Lemma \ref{thm: projecting curves in A3}. 

\begin{proof}[Proof of Lemma \ref{thm: projecting curves in A3}.] If $C$ is contained in a plane $H$, then we may simply drop one of the coordinates to obtain an affine plane curve $C'$ of degree $d$, birational to $C$ such that
\[
N_\mathrm{aff}(C; \ell)\leq N_\mathrm{aff}(C; \ell).
\]
So assume that $C$ is not contained in any plane. Let $Z$ be the projective closure of $C$ in $\PP^3$. Denote by $H_\infty$ the plane at infinity and fix a point $q_0\in C$. The cone $S$ on $Z$ and $q_0$ is an irreducible surface of degree $\leq d-1$, since $C$ is not contained in a plane. Since $S$ does not contain $H_\infty$, we obtain via Lemma \ref{thm: coordinate transformation} a point $p=(p_1: p_2: p_3: 0)$ (represented by coprime $p_i\in \FF_q[t]$) in $H_\infty\setminus S$ with $|p_i|< d$. Let $H$ be the plane defined by $p_1x+p_2y+p_3z = 0$ in $\PP^3$ and consider the projection
\[
\pi: \PP^3 \to H: q\mapsto (p\cdot q)p - (p\cdot p)q.
\]
Since $p\not \in S$, the restriction $\pi|_Z$ is birational onto its image, and $\pi(Z)$ is a plane curve of degree $d$, see \cite[Sec.\,18]{harris}. If we work in $\AA^3$, the map
\[
\pi: \AA^3\to H: q\mapsto (p\cdot q)p - (p\cdot p)q
\]
is a rescaling of the projection above. Moreover, $\pi|_C$ is still birational onto its image, and after dropping a coordinate, $\pi(C)=C'$ is an affine plane curve of degree $d$. Let $\tilde{C}\to C$ be the normalization of $C$. Then $\tilde{C}\to C'$ is an isomorphism away from the singular points of $C'$ and by \cite[Thm.\,17.7(b)]{kunz} the singular points of $C'$ have at most $(d-1)(d-2)$ preimages under $\tilde{C}\to C'$. In particular, both these claims also apply to $C\to C'$. Finally, note that $h(\pi(q))\leq h(q)+2h(p)<h(q)+2\log_q d$ to conclude that
\[
N_\mathrm{aff}(C; \ell)\leq N_\mathrm{aff}(C'; \ell+2\log_q d) +d^2. \qedhere 
\]

\end{proof}

\bibliographystyle{alpha}
\bibliography{bibfile}

\def\cprime{$'$}
\begin{thebibliography}{CCDN19}

\bibitem[BHBS06]{Brow-Heath-Salb}
T.~D. Browning, D.~R. Heath-Brown, and P.~Salberger.
\newblock Counting rational points on algebraic varieties.
\newblock {\em Duke Math. J.}, 132(3):545--578, 2006.

\bibitem[BP89]{Bombieri-Pila}
E.~Bombieri and J.~Pila.
\newblock The number of integral points on arcs and ovals.
\newblock {\em Duke Math. J.}, 59(2):337--357, 1989.

\bibitem[BV83]{Bombi-Vaal}
E.~Bombieri and J.~Vaaler.
\newblock On {S}iegel's lemma.
\newblock {\em Invent. Math.}, 73(1):11--32, 1983.

\bibitem[CCDN19]{CCDN}
W.~{Castryck}, R.~{Cluckers}, P.~{Dittmann}, and K.~{Nguyen}.
\newblock {The dimension growth conjecture, polynomial in the degree and
  without logarithmic factors}.
\newblock {\em arXiv e-prints}, page arXiv:1904.13109, Apr 2019.

\bibitem[CFL19]{cluckers2019uniform}
R.~{Cluckers}, A.~{Forey}, and F.~{Loeser}.
\newblock {Uniform Yomdin-Gromov parametrizations and points of bounded height
  in valued fields}.
\newblock {\em arXiv e-prints}, page arXiv:1902.06589, February 2019.
\newblock To appear in Algebra and Number Theory.

\bibitem[CM06]{CafureMatera}
A.~Cafure and G.~Matera.
\newblock Improved explicit estimates on the number of solutions of equations
  over a finite field.
\newblock {\em Finite Fields and Their Applications}, 12:155--185, 2006.

\bibitem[EV05]{Ellenb-Venkatesh}
J.~Ellenberg and A.~Venkatesh.
\newblock On uniform bounds for rational points on nonrational curves.
\newblock {\em Int. Math. Res. Not.}, (35):2163--2181, 2005.

\bibitem[Gao01]{GAOpdes}
S.~Gao.
\newblock Factoring multivariate polynomials via partial differential
  equations.
\newblock {\em Mathematics of Computation}, 72, 05 2001.

\bibitem[Har95]{harris}
J.~Harris.
\newblock {\em Algebraic Geometry, A first course}.
\newblock Graduate texts in mathematics, Spring-Verlag, 1995.

\bibitem[HB83]{Heath-Brown-cubic}
D.~R. Heath-Brown.
\newblock Cubic forms in ten variables.
\newblock {\em Proc. London Math. Soc. (3)}, 47(2):225--257, 1983.

\bibitem[HB02]{Heath-Brown-Ann}
D.~R. Heath-Brown.
\newblock The density of rational points on curves and surfaces.
\newblock {\em Ann. of Math. (2)}, 155(2):553--595, 2002.

\bibitem[Jou83]{bertinithm}
J.P. Jouanolou.
\newblock {\em Th{\'e}or{\`e}mes de Bertini et applications}.
\newblock S{\'e}ries de math{\'e}matiques pures et appliqu{\'e}es.
  Universit{\'e} Louis Pasteur, 1983.

\bibitem[Kal95]{KALTOFEN1995}
E.~Kaltofen.
\newblock Effective {N}oether irreducibility forms and applications.
\newblock {\em J. Comput. System Sci.}, 50(2):274--295, 1995.
\newblock 23rd Symposium on the Theory of Computing (New Orleans, LA, 1991).

\bibitem[Kun05]{kunz}
E.~Kunz.
\newblock {\em Introduction to plane algebraic curves}.
\newblock Birkh\"{a}user Boston, 2005.

\bibitem[Rup86]{RuppertCrelle}
W.~Ruppert.
\newblock Reduzibilit\"{a}t ebener {Kurven}.
\newblock {\em J. reine angew. Math.}, 369:167--191, 1986.

\bibitem[Sal13]{Salberger-dgc}
P.~Salberger.
\newblock Counting rational points on projective varieties.
\newblock preprint from 2009, version from 2013.

\bibitem[Sed17]{sedunova}
A.~Sedunova.
\newblock On the bombieri-pila method over function fields.
\newblock {\em Acta Arith.}, 181(4):321–--331, 2017.

\bibitem[Ser89]{Serre-Mordell}
J.-P. Serre.
\newblock {\em Lectures on the {M}ordell-{W}eil theorem}.
\newblock Aspects of Mathematics, E15. Friedr. Vieweg \& Sohn, Braunschweig,
  1989.
\newblock Translated from the French and edited by Martin Brown from notes by
  Michel Waldschmidt.

\bibitem[Wal15]{Walsh}
M.~N. Walsh.
\newblock Bounded rational points on curves.
\newblock {\em Int. Math. Res. Not. IMRN}, (14):5644--5658, 2015.

\end{thebibliography}

\noindent \textsc{Section of Algebra, Department of Mathematics, KU Leuven}

\noindent \textsc{Celestijnenlaan 200B, 3001 Leuven (Heverlee), Belgium}

\noindent \vspace{-0.35cm}

\noindent \texttt{floris.vermeulen@kuleuven.be}

\noindent \texttt{https://sites.google.com/view/floris-vermeulen/homepage}

\end{document}